\documentclass[11pt, reqno]{amsart}
\usepackage{graphicx, amssymb, amsmath, amsthm, color, slashed}
\numberwithin{equation}{section}

\usepackage{hyperref}
\usepackage{amscd}
\usepackage{tikz}
\usepackage{tkz-euclide}
\usepackage{graphicx}
\usepackage{geometry}

\let\th=T

\newcommand{\R}{\Bbb R}
\newcommand{\N}{\mathbb N}

\def\({\left(}
\def\){\right)}
\def\<{\left\langle}
\def\>{\right\rangle}
\let\Re=\undefined\DeclareMathOperator*{\Re}{Re}

\newtheorem{theorem}{Theorem}[section]

\newtheorem{lemma}[theorem]{Lemma}

\newtheorem{corollary}[theorem]{Corollary}

\newtheorem{proposition}[theorem]{Proposition}

\theoremstyle{definition}
\newtheorem{definition}[theorem]{Definition}
\newtheorem{remark}[theorem]{Remark}

\theoremstyle{remark}

\begin{document}
	
	\title[Nonlinear Schr\"odinger equations]{ \bf Generic ill-posedness for  Schr\"odinger equation with power-type  nonlinearity on $\Bbb{S}^2$ }
	
	\author{Sijie Qian}
\address{Sijie Qian
\newline \indent The Graduate School of China Academy of Engineering Physics, Beijing 100088,\ P. R. China}
\email{qiansijie22@gscaep.ac.cn}
	\author{Yilin Song}
\address{Yilin Song
\newline \indent The Graduate School of China Academy of Engineering Physics, Beijing 100088,\ P. R. China}
\email{songyilin21@gscaep.ac.cn}
	
\author{Ruixiao Zhang}
\address{Ruixiao Zhang
\newline \indent The Graduate School of China Academy of Engineering Physics, Beijing 100088,\ P. R. China}
\email{zhangruixiao21@gscaep.ac.cn}
	
\author{Jiqiang Zheng}
\address{Jiqiang Zheng
\newline \indent Institute of Applied Physics and Computational Mathematics, Beijing, 100088, China.
\newline\indent
National Key Laboratory of Computational Physics, Beijing 100088, China}
\email{zheng\_jiqiang@iapcm.ac.cn, zhengjiqiang@gmail.com}

\begin{abstract}In this article, we investigate the local well-posedness of the nonlinear Schr\"odinger equation on the two-dimensional sphere $\mathbb{S}^2$:
\begin{align*}
i\partial_tu+\Delta_{g}u=F(u).
\end{align*}
The nonlinearity $F(u)$ is assumed to be gauge-invariant. More presicely, there exists a  function $V\in C^\infty(\Bbb C,\Bbb R)$ such that $F=\frac{\partial V}{\partial \bar{z}}$. Moreover, $V(z)$ obeys 
\begin{gather}\label{H-11}
  V(e^{i\theta}z)=V(z),\,\,\theta\in\Bbb R,\,\,z\in\Bbb C,
  |\partial_z^{k_1}\partial_{\bar{z}}^{k_2}V(z)|\leq C_{k_1,k_2}(1+|z|)^{1+\alpha-k_1-k_2},\tag{H-1}%
\end{gather}
for some $\alpha\geq3.$
The main  contribution of this paper is  the new lower bound of  threshold of local well-posedness $s_c(\Bbb S^2,\alpha)$.   Specifically, under assumption \eqref{H-11}, we prove that for $\alpha \geq 3$, the equation is ill-posed in $H^s(\mathbb{S}^2)$ with $s <  1 - \frac{2}{\alpha-1}$ in the sense that the norm inflation occurs. Combined with the well-posedness  in  Yang [Sci. China Math. 58 (2015), 1023-1046], the exact threshold $s_c(\Bbb S^2,\alpha)$ for $\alpha\geq5$ is $1-\frac{2}{\alpha-1}$, which matches the scaling-critical regularity as the Euclidean setting.   Moreover, for $\alpha \in [3, \frac{11}{3})$, we show that the solution map is not uniformly continuous  in the range $0 < s < \frac14$ for the power-type nonlinearity $F(u)=|u|^{\alpha-1}u$, which lies strictly above the scaling-invariant threshold.  This provides a new characterization of the ill-posedness regime for all $\alpha \geq 3$, extending an earlier result of Burq-G\'erard-Tzvetkov [Math. Res. Lett. 9 (2002), 323-335]. Our result can also be regarded as a Schr\"odinger counterpart of Xia [Int. Math. Res. Not. (2021), 15533-15554].
	\bigskip

\noindent \textbf{Mathematics Subject Classification (2020)} Primary: 35Q55; Secondary: 35L05.
\end{abstract}
	
\maketitle

\section{Introduction}
\subsection{Background and motivation}
We consider the Cauchy problem of the nonlinear Schr\"odinger equation(NLS)
\begin{equation}\label{fml-NLS}
\begin{cases}
(i\partial_{t}+\Delta_g)u=F(u) , &(t,x)\in\Bbb{R}\times M, \\
u(0,x)=u_{0}(x)\in H^s(M),
\end{cases}
\end{equation}
where $u:I\times M\to\Bbb C $ is the unknown function. The function $F$ satisfies $F(0)=0$. Then there exists a gauge-invariant function $V\in C^\infty(\Bbb C,\Bbb R)$ such that $F=\frac{\partial V}{\partial \bar{z}}$. Moreover, $V(z)$ obeys 
\begin{gather}\label{condition-2}
 V(e^{i\theta}z)=V(z),\,\,\theta\in\Bbb R,\,\,z\in\Bbb C,
  |\partial_z^{k_1}\partial_{\bar{z}}^{k_2}V(z)|\leq C_{k_1,k_2}(1+|z|)^{1+\alpha-k_1-k_2}.
\end{gather}
The number $\alpha$ usually measures the growth of nonlinear term and refers to the degree of $F$.
	
The nonlinear Schr\"odinger equation \eqref{fml-NLS} has been extensively investigated over the decades on various manifolds, including Euclidean space $\mathbb{R}^d$, the torus $\mathbb{T}^d$, and the sphere $\mathbb{S}^d$ (see, e.g., \cite{Bour,Burq1,Burq2,Caze,HTT-Duke,KM,KV1,Pausader,Tao} and the references therein). Here, we focus on the case where $M = \mathbb{S}^2$, the two-dimensional sphere endowed with the canonical metric $g$.
In the local coordinate system of $M$, the Laplace-Beltrami operator takes the form
\begin{align*}
\Delta_g=\dfrac{1}{\sqrt{|g|}}\sum\limits_{i,j=1}^2\dfrac{\partial}{\partial x_i}\Big(g^{ij}\sqrt{|g|}\dfrac\partial{\partial x_j}\Big).
\end{align*}

The Cauchy problem \eqref{fml-NLS} with $M=\Bbb S^2$ has a Hamiltonian structure and its solution satisfies the following two conservation laws:
\begin{gather}\label{mass}
M(u)=\int_{\Bbb{S}^2}\left|u\right|^2\,dx \equiv M(u)(0),  \\
E(u)=\frac12\int_{\Bbb{S}^2}\left|\nabla_g u\right|^2\,dx+{\int_{\Bbb{S}^2}V(u)\,dx} \equiv E(u)(0).\label{energy}
\end{gather}
	
When $M=\R^2$, the equation \eqref{fml-NLS} with $F(u)=|u|^{\alpha-1}u$ is invariant under the following scaling transformation
\begin{align*}
u_\lambda(t,x)=\lambda^\frac{2}{\alpha-1}u(\lambda^2t,\lambda x).
\end{align*}
A direct computation shows that $\|u_\lambda(0)\|_{\dot H^{s_c}(\mathbb{R}^2)} = \|u(0)\|_{\dot H^{s_c}(\mathbb{R}^2)}$ for $s_c := 1 - \frac{2}{\alpha-1}$. The problem is called $\dot H^{s_c}$-critical if $s=s_c$;  $\dot H^{s_c}$-subcritical if $s>s_c$;  $\dot H^{s_c}$-supercritical if $s<s_c$. For other nonlinearities with the same degree as $|u|^{\alpha-1}u$, we also refer  to $s_c$ as the critical regularity. 	
	
Next, we recall the definitions of local uniform well-posedness and ill-posedness for \eqref{fml-NLS} on $M$.
\begin{definition}[Local uniform well-posedness,\cite{CCT-03}]\label{def-well-pose}
The Cauchy problem \eqref{fml-NLS} is  well-posed in $H^s(M)$ if for any bounded subset $B$ in $H^s(M)$, there exists $T>0$ such that the solution map
\begin{align*}
u_0\in C^\infty(M)\cap B\mapsto u\in C(\left[-T,T\right],H^s(M))
\end{align*} 
is uniformly continuous when the source space is endowed with the $H^s$ norm and when the target space is endowed with
\begin{equation*}
\Vert u\Vert_{C(\left[-T,T\right],H^s(M))}=\sup_{\left|t\right|\leqslant T}\Vert u(t)\Vert_{H^s(M)}.
\end{equation*} 
\end{definition}
In contrast to the local well-posedness, we define two kinds of ill-posedness.
\begin{definition}[Not uniformly continuous]\label{def-ill-pose-weak}
The solution map of \eqref{fml-NLS} is not uniformly continuous in $H^s(M)$ if for any $0<\delta<1$ and $t>0$, there exist solutions $u,v$ to \eqref{fml-NLS} with initial data $u_0,v_0$ such that the following statements hold:
\begin{equation*}
\begin{aligned}
\|u_0\|_{H^s(M)} + \|v_0\|_{H^s(M)} &\lesssim 1,\\
\|u_0-v_0\|_{H^s(M)} &\lesssim \delta,\\
\sup_{0<\tau<t}\|u(\tau)-v(\tau)\|_{H^s(M)} &\gtrsim 1.
\end{aligned}
\end{equation*} 
\end{definition}
Next we define a stronger instability phenomenon describing  the rapid growth of norms.

\begin{definition}[Norm inflation]\label{def-strong-ill-posedness}
Let $s>0$, for any $\epsilon,t>0$, there exists solution $u$ to \eqref{fml-NLS} with initial data $u_0\in H^s(M)$ such that the following hold,
\begin{equation*}
\begin{gathered}
\|u_0\|_{H^s(M)} \lesssim \epsilon,\\
\sup_{0<\tau<t}\|u(\tau)\|_{H^s(M)} \gtrsim \epsilon^{-1}.
\end{gathered}
\end{equation*}
\end{definition}
	
Since $u\equiv 0$ solves \eqref{fml-NLS} with $u_0=0$, norm inflation implies that the solution map is not uniformly continuous directly. Therefore, it can be regarded as a stronger notion of ill-posedness.

Let us briefly review the well-posedness theory for \eqref{fml-NLS} on different manifolds $M$, where the underlying geometry plays a crucial role. In the study of local well-posedness, Strichartz estimates  
\begin{align}\label{Strichartz}
\big\|e^{it\Delta_g}u_0\big\|_{L_t^qL_x^r(I\times M)}\lesssim \|u_0\|_{H^s(M)}
\end{align} 
are crucial in establishing the local well-posedness, where $s\geq0$ and $q,r\in\R$. When $M=\R^2$, the Strichartz estimates hold  globally in time  for $s=0$ and $(q,r)$ satisfying the following admissible condition
\begin{align}\label{admissible}
\frac{2}{q}=2\left(\frac12-\frac1r\right),\,\,2<q<\infty,\,\,2\leq  r<\infty.
\end{align}  
These estimates rely heavily on the dispersive estimate
\begin{align*}
\big\|e^{it\Delta}u_0\big\|_{L_x^\infty(\R^2)}\lesssim|t|^{-1}\|u_0\|_{L^1(\R^2)}.
\end{align*}
Combining Strichartz estimates   with the fractional chain rule, Cazenave-Weissler  \cite{Caze} established the $ H^{s}$ local well-posedness for $s\geq \max\{0,s_c\}$ and $s\leq 1+\alpha$.   This result is sharp: for $s < s_c$, Christ-Colliander-Tao \cite{CCT-03} showed that the equation \eqref{fml-NLS} is  ill-posed in $H^s(\R^2)$ with $s<\max\{0,s_c\}$ and $\alpha \leq 1+k$ where $k\geq2$ is an integer. For the global well-posedness and scattering for cubic NLS on $\R^2$, we refer to Killip-Tao-Visan \cite{KVZ-JEMS} and Dodson \cite{Dodson}.
	
On compact manifolds, however, the dispersive estimate fails in general due to the presence of trapped geodesics, which inhibit dispersion over long times. For the  torus $M=\Bbb T^2$, Bourgain \cite{Bour} utilized the circle method with the  Stein-Tomas argument to derive the $L^p$-Strichartz estimate for $p \geq 4$:
\begin{align}\label{tori}
\big\|e^{it\Delta_{\Bbb T^2}}P_Nf\big\|_{L_{t,x}^p(I\times\Bbb T^2)}\lesssim N^{1-\frac{4}{p}+\varepsilon}\|f\|_{L^2(\Bbb T^2)},
\end{align}
where   $P_N$ is the spectral projector onto  frequencies near $N$. Notice that when $p>4$, the additional derivative loss will not occur.  Combining this estimate with the $X^{s,b}$ space, he proved the local well-posedness for \eqref{fml-NLS} with initial data $u_0\in H^1(\Bbb T^2)$ with $\alpha\in2\Bbb N+1$. In the cubic case ($\alpha=3$), one can obtain the local well-posedness in $H^s$ for arbitrary $s>0$. On the other hand,  Kishimoto \cite{Kishimoto} showed the solution map is not $C^3$ in $L^2$. On irrational tori  $\Bbb T_\theta^2:=\prod_{j=1}^2\R/(\alpha_j\Bbb Z)$, Guo-Oh-Wang \cite{GTW} obtained similar results to those in \cite{Bour}. For the hyperbolic Schr\"odinger case, i.e. $\Delta=\partial_{x_1}^2-\partial_{x_2}^2$, Wang \cite{Wang2} proved the sharp local well-posedness of cubic hyperbolic NLS for $u_0\in H^s(\Bbb T^2)$ with $s>\frac{1}{2}$. Very recently, Shen-Wang \cite{Shen} proved the norm inflation phenomenon for $s\leq\frac12$ with $s\neq0$. In the semi-periodic setting $\mathbb{R} \times \mathbb{T}$, Takaoka-Tzvetkov \cite{TT} established the $L_{t,x}^4$ Strichartz estimates without loss of derivatives and then showed the  local well-posedness for initial data $u_0\in L^2(\R\times\Bbb T)$, differing from the purely periodic case.
	
For general compact manifolds without boundary, Burq-G\'erard-Tzvetkov \cite{Burq1} constructed a parametrix for $e^{it\Delta_g}$ and derived a semiclassical dispersive estimate. Applying the abstract Strichartz estimate of Keel-Tao, the following holds for $(q,r)$ satisfying \eqref{admissible}
\begin{align}\label{Stri-M}
\big\|e^{it\Delta_g}f\big\|_{L_t^qL_x^r(I\times M)}\leq C(I,M)\|f\|_{H^\frac1q(M)}, 
\end{align} 
which incurs a derivative loss of order $\frac1q$. Compared to the Euclidean space, the Strichartz estimate must be local in time. As an application, they showed that the solution to \eqref{fml-NLS} is locally well-posed in $H^s(M)$ with $s>1-\frac{1}{\alpha-1}$.  On the sphere $\mathbb{S}^2$, however, the loss in \eqref{Stri-M} is not  sharp. In fact, for $p = q = 4$, they improved the loss to $\frac18$, and showed that the $\frac{1}{8}$-order  loss of derivative is optimal in \cite{Burq1}. In \cite{Burq2}, they further established bilinear eigenfunction estimates capturing high-low frequency interactions, leading to local well-posedness for the cubic NLS on $\mathbb{S}^2$ for $s > \frac14$. This 
result is sharp up to the endpoint: for $s < \frac14$, the flow map fails to be uniformly continuous. For sub-cubic nonlinearity, that is $|u|^{\alpha-1}u$ with $\alpha<3$, they also proved that the solution map is not uniformly continuous in $H^s(\Bbb S^2)$ with $s<\frac{1}{4}$. For more details, we refer to \cite{GerardICM}. However, there are few results concerning the instability between cubic and quintic. For general odd nonlinearities $\alpha \in 2\N+1$ with $\alpha \geq 5$, Yang \cite{Yang} proved local well-posedness for $s > 1 - \frac{2}{\alpha-1}$, matching the scaling exponent. For the nonlinear wave equation with general power, Burq-Tzvetkov \cite{Burq-Tzvektov} and Xia \cite{Xia} constructed the solution exhibiting the norm inflation phenomenon in the super-critical regime.

Let us now discuss more details on the instability of cubic NLS on spheres.  In \cite{Burq2001}, it was shown that the solution map to \eqref{fml-NLS}  with $\alpha=3$ is not uniformly continuous in $H^s(\mathbb{S}^2)$ for $0 \leq s < \frac14$. Banica \cite{Banica} later gave a more concrete instability construction. These instabilities are illustrated using highest-weight spherical harmonics concentrating near the equator:
\begin{align}
\phi_k(x)=k^{\frac{1}{4}-s}(x_1+ix_2)^k,\quad\text{ where }(x_1,x_2,x_3)\in\Bbb{S}^2,\forall k\in\Bbb{N}.
\end{align}
As shown in \cite{Han-JGA}, these functions maximize the $L^p$ norm (for $2 < p \le 6$) among $L^2$-normalized eigenfunctions in $E_k(\mathbb{S}^2)$, the $k$-th eigenspace of $-\Delta_{\mathbb{S}^2}$. Therefore, for the cubic nonlinearity,  using $\phi_k$ as initial data leads to ill-posedness.
	
For $p \geq 6$, however, the maximizer of $\frac{\|e_n\|_{L^p}}{\|e_n\|_{L^2}}$  is the zonal spherical harmonic function. The highest-weight functions $\phi_k$ is no longer sufficient  to prove ill-posedness below the scaling index $s_c = 1 - \frac{2}{\alpha-1} > \frac14$. To fully exploit the concentration phenomenon, one need to use the zonal spherical harmonic functions. Indeed, after a suitable scaling transformation
\begin{equation*}
\varphi(x) \rightarrow n^a\varphi(nx),
\end{equation*} 
of bump function $\varphi$ which is supported near the north pole, we obtain an $L^2$ normalized bump function which is supported in the region of radial $n^{-1}$ instead. Thus, it will simplify the analysis.

\begin{remark}
On higher-dimensional spheres $\Bbb S^d, d\ge3$, the maximizers of $\|e_n\|_{L^4(\Bbb S^d)}/\|e_n\|_{L^2(\Bbb S^d)}$ are always zonal spherical harmonic functions. Thus, in this case, using zonal spherical harmonic functions as initial data helps us to obtain sharp ill-posedness results for cubic NLS. For energy-supercritical NLS on $\mathbb{S}^3$, we refer to Burq-G\'erard-Tzvetkov \cite{Burq3} for ill-posedness in $H^1(\Bbb S^3)$.
\end{remark}

\subsection{Main result}

The main goal of this article is to study the instability for equation \eqref{fml-NLS}. More precise, we establish the ill-posedness for rough initial data. 

The first theorem reveals that for all $\alpha\geq3$, the norm inflation will occur for $0<s<s_c$ which is similar to the case of \cite{CCT-03}. 
\begin{theorem}[Strong ill-posedness]\label{thm-illpose}
Let $F(u) $ satisfy the condition \eqref{condition-2} with $\alpha\ge 3$ and $s<s_c=1-\frac{2}{\alpha-1}$, then the Cauchy problem \eqref{fml-NLS} is strongly ill-posed in $H^s(\Bbb S^2)$. In other word, there exist a positive sequence $\{t_n\}_{n\in\mathbb{N}}$ satisfying $t_n\to 0$ as $n\to\infty$ and a sequence of smooth function $\{u_n(t)\}_{n\in\mathbb{N}}$ such that the following statements hold:
\begin{enumerate}
\item $u_n(t)$ satisfies \eqref{fml-NLS} with initial data $u_n(0)$.
\item $\lim\limits_{n\to\infty}\|u_n(0)\|_{H^s(\Bbb S^2)} = 0$.
\item $\lim\limits_{n\to\infty}\|u_n(t_n)\|_{H^s(\Bbb S^2)} = \infty$ \label{item-unbound}.
\end{enumerate}
\end{theorem}
	
The second one is that when $\alpha\in[3,\frac{11}{3})$, the solution map is not uniform continuous with respect to time.  
\begin{theorem}[Ill-posedness]\label{thm-illpose2}
Let $F(u) = |u|^{\alpha-1}u$, $3\leq \alpha< \frac{11}{3}$ and $\frac{1}{4}-\frac{1}{2(\alpha-1)}<s<\frac{1}{4}$, then the Cauchy problem \eqref{fml-NLS} is ill-posed in $H^s(\Bbb S^2)$. In other words, there exists a positive sequence $\{t_n\}_{n\in\mathbb{N}}$ satisfies $t_n\to 0$ as $n\to\infty$ and two sequences of smooth functions $\{u_n(t)\}_{n\in\mathbb{N}},\{v_n(t)\}_{n\in\mathbb{N}}$ such that the following statements hold.
\begin{enumerate}
\item $u_n(t)$ and $v_n(t)$ satisfies \eqref{fml-NLS} with initial data $u_n(0)$ and $v_n(0)$, respectively.
\item $\lim\limits_{n\to\infty}\|u_n(0)-v_n(0)\|_{H^s(\Bbb S^2)} = 0$.
\item $\limsup\limits_{n\in \mathbb{N}}\|u_n(t_n)-v_n(t_n)\|_{H^s(\Bbb S^2)} \ge \frac{1}{2}$. 
\end{enumerate}
\end{theorem}
	
\begin{remark}
If $\alpha\ge 5,\alpha\in 2\mathbb{N}+1$, combined with the positive result obtained by Yang \cite{Yang}, Theorem \ref{thm-illpose} is sharp up to the end-point.  We also note that  the remaining cases in Theorem \ref{thm-illpose2} can be contained in Theorem \ref{thm-illpose} if we do not  distinguish $\langle u\rangle^{\alpha-1}u$ and $|u|^{\alpha-1}u$. 
		
Consequently, when $|u|^{\alpha-1}u$ and $\langle u\rangle^{\alpha-1}u$ are not distinguished, the full characterization of well and ill-posedness for \eqref{fml-NLS} with $\alpha\ge 3$ is given by the following figure:
\begin{center}
			\begin{tikzpicture}[scale=1] 
				\draw[->] (-7,0) -- (6,0) node[anchor=north] {$\alpha$};
				\draw[->] (-6,-2) -- (-6,7)  node[anchor=east] {$s$};
				\draw (-6,0) node[below left]{O};
				
				\draw[dotted] (-6,6) -- node [left=6cm] {$s=1$} (6,6);
				\draw [dotted] (-2,0) -- (-2,2) ;
				
				\draw [thick, domain = -3: 6] plot (\x,{6-36/(11+\x)}) node [above=0.5cm,left] {$s_c = 1-\frac{2}{\alpha-1}$} ;
				\draw [dotted, domain = -5: -3] plot (\x,{6-36/(11+\x)});
				\draw [thick] (-3,3/2) -- node [left=2cm] {$s=\frac{1}{4}$} (-5,3/2); 
				\draw [-] (-5,7) -- (-5,0);
				\draw [dotted] (-6,3/2) -- (-5,3/2) ;
				\draw (-5,0) node [below=0.05cm] {$\alpha=3$};
				\draw (-3,0) node [below] {$\alpha=\frac{11}{3}$};
				\draw (1,0) node [below=0.05] {$\alpha=5$};
				\draw (-2,0) node [below=0.05cm] {$\alpha=4$};
				\draw [dotted] (-3,0) -- (-3,3/2);
				
				\draw (-6,5.5) node [right] {\textcolor{orange}{NUC}:Not uniformly continuous.};
				\draw (-6,5) node [right] {\textcolor{red}{NI}: Norm inflation.};
				\fill (-3,0) circle (1pt);
				\fill (1,0) circle (1pt);
				\fill (-5,0) circle (1pt);
				\fill (-2,0) circle (1pt);
				
				\draw [dotted] (1,0) -- (1,7) ;
				
				\draw (2,4) node {\textcolor{blue}{Well-posed}} ;
				\draw (-5,3/2) circle (1pt) ;
				\draw (0,3/2) node {\textcolor{red}{NI}} ;
				\draw (2,2) node {\textcolor{red}{NI}} ;
				\draw (-3.8,0.4) node {\textcolor{red}{NI}} ;
				\draw (-2.5,0.6) node {\textcolor{red}{NI}} ;
				\draw (-4.5,1) node {\textcolor{orange}{NUC}};
				\draw (-3,1.7)  ;
				
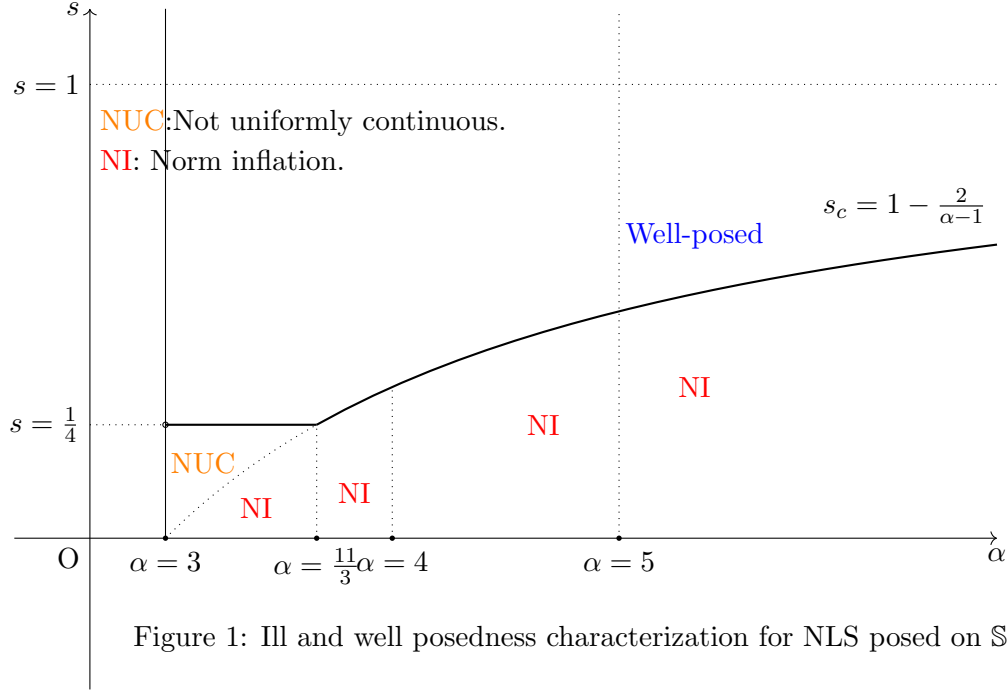
				\path (0,-1.3) node[xshift = 0.5cm](caption) {Figure 1: Ill and well posedness characterization for NLS posed on $\mathbb{S}^2$.} ;
			\end{tikzpicture}
		\end{center}
	\end{remark}
	
\begin{remark}
Compared to the Euclidean space, an interesting phenomenon is that the ill-posedness results exhibit a turning point at $\alpha=11/3$. In fact, this can be explained by combining the recent result obtained by Huang–Sogge \cite{Huang}. In their paper, they proved a sharp $L^{\alpha+1}_{t,x}$-type Strichartz estimates on Zoll manifolds with $\alpha \ge 1$:
\begin{equation}\label{fml-Stri-Lp}
\|e^{it\Delta_{\Bbb S^2}}u_0\|_{L^{\alpha+1}_{t,x}([0,1]\times \Bbb S^2)} \lesssim_\varepsilon
\|u_0\|_{H^{s_0(\alpha)+\varepsilon}(\Bbb S^2)},
\end{equation}
with
\begin{equation*}
s_0(\alpha)=\begin{cases}
\frac{1}{2}(\frac{1}{2}-\frac{1}{\alpha+1}),\,\, &\alpha\in(1,\frac{11}{3}),\\
1-\frac{4}{\alpha+1},\,\, &\alpha\in[\frac{11}{3},\infty].
\end{cases}
\end{equation*}
Moreover, $n\in \mathbb{N}$, $\beta\in C^\infty_0(\R)$ is a Littlewood-Paley cut off functions and denote $e_j(x)$ be the eigenfunction associated with the $j$-th eigenspace. They proved that the maximizer of \eqref{fml-Stri-Lp} is
\begin{equation*}
u_n=\begin{cases}
n^{\frac{1}{4}}(x_1+ix_2)^n,\,\, &\alpha\in(1,\frac{11}{3}),\\
\sum\limits_{j\in\mathbb{N}}n^{-1}\beta(n_j/n)e_j(x)\overline{e_j(x_0)},\,\,&\alpha\in[\frac{11}3{,\infty}].
\end{cases}
\end{equation*}
Here, $x_0$ denotes a fixed point in $\Bbb S^2$ and $n_j$ is the $j$-th eigenvalue associated with $\sqrt{-\Delta_{\mathbb S^2}}$.
		
In fact, the instability phenomenon is governed by the  potential energy $\|u\|^{\alpha+1}_{L^{\alpha+1}}$. Heuristically, on a short time interval,  the contribution of the Duhamel term is negligible in the sense of space-time averaging. Hence, we may approximate the solution 
$u$ by its linear part.
 Therefore, to find the initial data $u_0$ that most likely to cause instability, we require it to satisfy:
\begin{equation*}
\|e^{it\Delta_{\Bbb S^2}}u_0\|_{L^{\alpha+1}_{t,x}([0,1]\times \Bbb S^2)} \gtrsim \|u_0\|_{H^{s_0(\alpha)}(\Bbb S^2)}.
\end{equation*}
		
Therefore, we may reasonably conjecture that the turning point is induced by the properties of eigenfunctions on the spheres, rather technical reasons.
		
\end{remark}
	
	We end up this section by recalling some basic  definitions of Laplace-Beltrami operator and the associated Sobolev spaces.
	 Denote the spherical harmonic function by $Y_k^l$ , which is also known as the eigenfunction of Laplace-Beltrami operator on spheres, it  satisfies the following eigen-equation associated with the eigenvalue $\lambda_k=k(k+1)$,
\begin{align*}
	-\Delta_{\Bbb{S}^2}Y_{k}^l(x)=k(k+1)Y_{k}^l(x),\,-k\leq \ell\leq k,\quad\forall x\in\Bbb{S}^2.
\end{align*}
By using the spherical harmonic decomposition, we can decompose all $f\in L^2(\Bbb S^2)$ with
\begin{align*}
	f=\sum_{k\in\Bbb{Z}} {\pi_kf},
\end{align*}
where $\pi_k$ denotes the $k$-th projector on eigenspace $E_k=\operatorname{span}\{Y_{k}^\ell\}_{-k\leq\ell\leq k}$.

Let $s\geqslant0$, denote $H^s(\Bbb S^2)$ the Sobolev space associated with the operator $(Id-\Delta_g)^\frac{s}{2}$ equipped with the norm
\begin{align*}
	\Vert u\Vert_{H^s(\Bbb S^2)}=\Big(\sum_{k}\langle\lambda_k\rangle^{2s}\Vert \pi_ku\Vert_{L^2}^2\Big)^\frac{1}{2},
\end{align*}
where $\lambda_k=\sqrt{k(k+1)}$ and $\langle \cdot\rangle=(1+|\cdot|^2)^\frac12$.
	
\textbf{Organization of this paper}\hspace{2ex}The paper is organized as follows. In Section 2, we prove norm inflation phenomenon for \eqref{fml-NLS} which meets the scaling index when $\alpha>11/3$. In Section 3, we show the solution to \eqref{fml-NLS} is not uniformly continuous dependent in $H^s(\mathbb S^2)$ with $0<s<1/4$ when $3\leq\alpha<11/3$.
	
\subsection{Notations}
In this paper, we use $A\lesssim B$ to mean that there exists a constant such that $A\leqslant CB$, where the constant is not depending on $B$. We will also use $s+$ or $s-$, which means that there exists a small positive number such that $s+\varepsilon$ or $s-\varepsilon$ respectively.

	



\section{Norm inflation in regime $p\geq3$: the proof of Theorem \ref{thm-illpose}}
\begin{proof}[Proof of Theorem \ref{thm-illpose}]
Let $\kappa_n = (\log n)^{-\delta}$ for some $\delta > 0$, and define the initial data by $u_n(0) = \kappa_n n^{1-s} \varphi(nx)$, where $\varphi$ is a non-negative smooth function with compact support. Clearly,
\begin{equation*}
\|u_n(0)\|_{H^s(\Bbb S^2)} \sim \kappa_n \to 0,\quad \text{as}\;n\to\infty.
\end{equation*}
We set $f(z) := \langle z\rangle^{\alpha-1}$ and {$F(z)=\langle z\rangle^{\alpha-1}z$}. Then, the solution $v_n(t)$ to the ODE
\begin{equation*}
\begin{cases}
i\partial_t v_n = F(v_n),\\
v_n(0) = u_n(0),
\end{cases}
\end{equation*} 
is given explicitly by
\begin{equation*}
v_n(t) = \kappa_n n^{1-s} \varphi(nx) \exp\big(-itf(\kappa_n n^{1-s} \varphi(nx))\big).
\end{equation*}
		
For convenience, we define the time parameter by
\begin{equation}\label{fml-two-para}
t_n:=(\log n)n^{-(\alpha-1)(1-s)},\,\, 0<\delta<\frac{s}{1+s(\alpha-1)}.
\end{equation}
		
Performing the change of variable $y = nx$, we get the estimate  
\begin{align}\nonumber
\|\nabla_x v_n(t)\|_{L^2_x(\Bbb S^2)} &= \kappa_n n^{1-s} \big\|\nabla_y\big( \varphi(y)\exp\big( -itf(\kappa_n n^{1-s} \varphi(y)) \big) \big) \big\|_{L^2_y(\Bbb S^2)}\\\nonumber
&\ge \kappa_n n^{1-s} \Big( (\alpha-1) t \kappa_n^{\alpha-1} n^{(\alpha-1)-s)} \||\varphi(y)|^k\nabla_y \varphi(y)\|_{L^2_y(\Bbb S^2)} - \|\nabla_y \varphi(y)\|_{L^2_y(\Bbb S^2)} \Big)\\\label{fml-nabla-vn}
&\ge c(\alpha-1)t \kappa_n^{\alpha} n^{\alpha(1-s)} - C \kappa_n n^{1-s}.
\end{align}
{From the choice of $t_n,\kappa_n$ and $\delta$, one can verify that for sufficiently large $n$} 
\begin{equation*}
{c(\alpha-1)t \kappa_n^{\alpha} n^{\alpha(1-s)} \gg C \kappa_n n^{1-s},\,\, \forall \,\, 0<t<t_n.}
\end{equation*}
As a consequence, we have
\begin{equation*}
{\sup_{0<t<t_n}\|v_n(t)\|_{H^1_x(\Bbb S^2)} \sim \kappa_n^\alpha (\log n) n^{1-s}.}
\end{equation*}
Similarly, we also get
\begin{equation*}
{\sup_{0<t<t_n}\|v_n(t)\|_{H^2_x(\Bbb S^2)} \sim \kappa_n^{2\alpha-1} (\log n)^2 n^{2-s}.}
\end{equation*}
By using the interpolation inequality and \eqref{fml-two-para}, for all $0<s<1$ there exists $c=c(\alpha,\delta)>0$ such that
\begin{equation}\label{fml-vn-Hs}
\sup_{0<t<t_n}\|v_n(t)\|_{H^s(\Bbb S^2)} \ge \sup_{0<t<t_n}\|v_n(t)\|^{2-s}_{H^1(\Bbb S^2)}\cdot  \sup_{0<t<t_n}\|v_n(t)\|^{s-1}_{H^2(\Bbb S^2)} \sim (\log n)^c.
\end{equation}
		
Next, we aim to  show that $v_n(t)$ serves as a good approximation to the  solution $u_n(t)$. To this end, we introduce a modified energy   
\begin{equation*}
E_n(u) := \Big( n^{2s}\|u\|^2_{L^2(\Bbb S^2)} + n^{2(s-2)}\|\Delta_g u\|^2_{L^2(\Bbb S^2)} \Big)^{\frac{1}{2}}.
\end{equation*}
A direct computation shows that $\|u\|_{H^s(\mathbb{S}^2)} \lesssim E_n(u)$ for all $n \in \mathbb{N}$, and  
\begin{equation}\label{fml-deri-En}
\begin{aligned}
\frac{d}{dt} E^2_n(u) &= 2n^{2s} \int_{\Bbb S^2} \Re(u_t \cdot\overline{u})\, dx + 2n^{2(s-2)} \int_{\Bbb S^2} \Re(\Delta_g u_t \cdot\Delta_g\overline{u})\, dx\\
&\lesssim n^{2s} \|u_t\|_{L^2(\Bbb S^2)}\|u\|_{L^2(\Bbb S^2)} + n^{2(s-2)} \|\Delta_g u_t\|_{L^2(\Bbb S^2)}\|\Delta_g u\|_{L^2(\Bbb S^2)}.
\end{aligned}
\end{equation}
		
We claim that there exists $\epsilon>0$ such that  for all $t\in(0,t_n)$ there holds
\begin{equation}\label{fml-asser-energy-app}
\sup_{0<t<t_n}E_n(u_n(t)-v_n(t)) \lesssim n^{-\epsilon}.
\end{equation} 
Combining this with~\eqref{fml-nabla-vn} will  complete the proof of Theorem \ref{thm-illpose}.
		
It remains to show the claim \eqref{fml-asser-energy-app}. Let $w_n := u_n - v_n$, then it satisfies $w_n(0)=0$ and 
\begin{equation}\label{fml-wn-eq}
(i\partial_t + \Delta_g)w_n = F(u_n)-F(v_n)-\Delta_g v_n.
\end{equation}
By taking the derivatives on both sides of the above equation, we obtain
\begin{equation}\label{fml-nab-wn-eq}
(i\partial_t + \Delta_g)\Delta_g w_n =\Delta_g \big(F(u_n)-F(v_n)\big)-\Delta_g^2 v_n.
\end{equation}
For $m=1,2,\cdots,\lfloor \alpha-2\rfloor$, using the explicit formula of $v_n$, we have the pointwise bound
\begin{equation}\label{fml-v-deri-Linf}
\sup_{0<t<t_n}\|\nabla^m v_n(t)\|_{L^\infty(\Bbb S^2)} \lesssim n^{1-s+m} \big(\log n\big)^{m(1-\delta \alpha)}.
\end{equation}
The Gagliardo-Nirenberg inequality yields
\begin{equation}\label{fml-GN-ineq}
\|h\|_{L^\infty(\Bbb S^2)} \lesssim \|h\|^{\frac{1}{2}}_{H^2(\Bbb S^2)} \|h\|^{\frac{1}{2}}_{L^2(\Bbb S^2)} \lesssim n^{1-s} E_n(h).
\end{equation}
\eqref{fml-wn-eq} and the fact  
\begin{equation*}
F(u_n)-F(v_n) = \mathcal{O}(1+|v_n|^{\alpha-1}+|u_n|^{\alpha-1}),
\end{equation*}
imply 
\begin{align}\nonumber
n^{s}\| \mathcal{O}(1+|v_n|^{\alpha-1}+|u_n|^{\alpha-1}) w\|_{L^2(\Bbb S^2)} &\lesssim n^{s} \big( 1 + \|v_n\|^{\alpha-1}_{L^\infty(\Bbb S^2)} + \|u_n\|^{\alpha-1}_{L^\infty(\Bbb S^2)} \big) \|w_n\|_{L^2(\Bbb S^2)}\\\label{fml-En-T1}
&\lesssim n^{(1-s)(\alpha-1)} \big( E_n(w_n) + E^\alpha_n(w_n) \big).
\end{align}
		
Next, to estimate the difference term appeared in \eqref{fml-wn-eq}, we further expand it to 
\begin{equation*}
\begin{aligned}
\Delta_g \big(F(u_n)-F(v_n)\big) &= \Delta_g w_n \cdot \mathcal{O}\big( 1+|w_n|^{\alpha-1}+|v_n|^{\alpha-1} \big)\\
&\hspace{2ex}+ \nabla w_n \cdot \mathcal{O}\Big( \big( 1+|w_n|^{\alpha-2}+|v_n|^{\alpha-2} \big) \big( 1+|w_n|+|v_n|+|\nabla w_n|+|\nabla v_n| \big) \Big)\\
&\hspace{2ex}+ w_n \cdot \mathcal{O}\Big( \big( 1+|w_n|^{\alpha-3}+|v_n|^{\alpha-3} \big) \big( 1+|\nabla v_n|^2+(|w_n|+|v_n|)|\nabla^2 v_n| \big) \Big).
\end{aligned}
\end{equation*}
By Minkowski's inequality, we have 
\begin{equation*}
\begin{aligned}
&\big\|\Delta_g \big(F(u_n)-F(v_n)\big) \big\|_{L^2(\Bbb S^2)}\\
\lesssim& \Big\|\Delta w_n \cdot \mathcal{O}\big( 1+|w_n|^{\alpha-1}+|v_n|^{\alpha-1} \big) \Big\|_{L^2(\Bbb S^2)}\\
&+ \Big\|\nabla w_n \cdot \mathcal{O}\Big( \big( 1+|w_n|^{\alpha-2}+|v_n|^{\alpha-2} \big) \big( 1+|w_n|+|v_n|+|\nabla w_n|+|\nabla v_n| \big) \Big) \Big\|_{L^2(\Bbb S^2)}\\
&+ \Big\|w_n\cdot\mathcal{O}\Big( \big( 1+|w_n|^{\alpha-3}+|v_n|^{\alpha-3} \big) \big( 1+|\nabla v_n|^2+(|w_n|+|v_n|)|\nabla^2 v_n| \big) \Big)\Big\|_{L^2(\Bbb S^2)}\\
:=& I_1+I_2+I_3.
\end{aligned}
\end{equation*}
Thanks to \eqref{fml-v-deri-Linf} and \eqref{fml-GN-ineq}, we get a   bound similar  to \eqref{fml-En-T1}:
\begin{equation*}
n^{s-2}(I_1+I_3) \lesssim n^{(1-s)(\alpha-1)}(\log(n))^{2(1-\delta \alpha)} \big( E_n(w_n) + E^\alpha_n(w_n) \big).
\end{equation*}
For the term in $I_2$ involving $|\nabla w_n|^2$, it can be bounded by
\begin{equation*}
\begin{aligned}
&\big( 1 + \|v_n\|^{\alpha-2}_{L^\infty(\Bbb S^2)} + \|u_n\|^{\alpha-2}_{L^\infty(\Bbb S^2)} \big) \|\nabla w_n\|_{L^4(\Bbb S^2)}^2\\
\lesssim& \big( 1+n^{(\alpha-2)(1-s)}+n^{(\alpha-2)(1-s)} E^{\alpha-2}_n(w_n) \big) \|\nabla w_n\|^2_{H^{\frac{1}{2}}(\Bbb S^2)} \\
\lesssim& \big( 1+n^{(\alpha-2)(1-s)}+n^{(\alpha-2)(1-s)} E^{\alpha-2}_n(w_n) \big) n^{\frac{3}{2}-2s} E^2_n(w_n)\\
\lesssim& n^{\alpha(1-s)+1} ( E_n^2(w_n) + E_n^\alpha(w_n) ) .
\end{aligned}
\end{equation*}
Similarly, we get
\begin{equation}\label{fml-En-T2}
n^{s-2}\big\|\Delta_g \big(F(u_n)-F(v_n)\big) \big\|_{L^2(\Bbb S^2)} \lesssim n^{(\alpha-1)(1-s)} (\log(n))^{2(1-\delta \alpha)} ( E_n^2(w_n) + E_n^\alpha(w_n) ).
\end{equation}
		
For the source terms $v_n$, we have 
\begin{equation}\label{fml-En-T3}
n^s\|\Delta_g v_n\|_{L^2(\Bbb S^2)} + n^{s-2}\|\Delta_g^2 v_n\|_{L^2(\Bbb S^2)} \lesssim n^2 (\log n)^{4(1-\delta \alpha)}.
\end{equation}
Since $E_n(w_n(0)) = 0$ and $E_n(w_n)(t)$ is continuous with respect $t$, we have $E_n(w_n(t)) \leq 1$ for $t$ sufficiently small.  
		
Therefore, we combine \eqref{fml-En-T1}, \eqref{fml-En-T2} with \eqref{fml-En-T3} to obtain the energy increment estimate:
\begin{equation*}
\begin{aligned}
\frac{d}{dt} E_n^2(w_n) &\lesssim n^{(1-s)(\alpha-1)}(\log(n))^{2(1-\delta \alpha)}( E_n^2(w_n) + E_n^\alpha(w_n) ) + n^2 (\log n)^{4(1-\delta \alpha)}E_n(w_n(t))\\
&\lesssim n^{(1-s)(\alpha-1)}(\log(n))^{2(1-\delta \alpha)}( E_n^2(w_n) + E_n^\alpha(w_n) ) \\
&\hspace{4ex} + \Big( n^{\frac{(1-s)(\alpha-1)}{2}}(\log n)^{1-\delta \alpha} E_n(w_n) \cdot n^{2-\frac{(1-s)(\alpha-1)}{2}} (\log(n))^{3(1-\delta \alpha)}\Big)\\
&\lesssim n^{(1-s)(\alpha-1)}(\log(n))^{2(1-\delta \alpha)} E_n^2(w_n)  + \frac{n^4(\log n)^{8(1-\delta \alpha)}}{n^{(1-s)(1-\alpha)} (\log n)^{2(1-\delta \alpha)} }.
\end{aligned}
\end{equation*}
		
Let 
\begin{equation*}
\widetilde{E}_n(w_n) := \exp{\big(-t n^{(1-s)(\alpha-1)} (\log n)^{2(1-\delta \alpha)}\big) } E^2_n(w_n),
\end{equation*}
then the energy estimate becomes
\begin{equation*}
\frac{d}{dt} \widetilde{E}_n(w_n) \lesssim n^{4-(1-s)(\alpha-1)}(\log n)^{6(1-\delta \alpha)} \exp{\big(t n^{(1-s)(\alpha-1)} (\log n)^{2(1-\delta \alpha)}\big) }.
\end{equation*}
Integrating from $0$ to $t$ on both sides gives
\begin{equation*}
E_n(w_n(t)) \lesssim n^{2-(1-s)(\alpha-1)} (\log n)^{2(1-\delta \alpha)} \exp{\big(t n^{(1-s)(\alpha-1)} (\log n)^{2(1-\delta \alpha)}\big) }.
\end{equation*}
Since $0\le t \le t_n$ and $s<1-\frac{2}{\alpha-1}$, for sufficiently large $n$ there exists $\epsilon>0$ such that
\begin{equation*}
\sup_{0<t<t_n}E_n(w_n(t)) \lesssim n^{-\epsilon}.
\end{equation*}
		
Recall the lower bound we have obtained in \eqref{fml-vn-Hs}, then it follows that:
\begin{equation*}
\sup_{t<t<t_n}\|u_n(t)\|_{H^s(\Bbb S^2)} \gtrsim (\log n)^c - n^{-\epsilon} \to \infty,
\end{equation*}
as $n\to \infty$. So far, we have obtained \eqref{item-unbound} in Theorem \ref{thm-illpose}. 
		
Thus, we conclude the proof of Theorem \ref{thm-illpose}. 
\end{proof}
	


\section{Not uniformly continuous: the proof of Theorem \ref{thm-illpose2}}
In this section, we show that the solution map is not uniformly continuous for $s<\frac14$, which extends the work of \cite{Burq2001}.
	
For $n\ge0$, we denote by {$\pi_n$  the spectral projection on eigenspace $E_n$. For a function $f \in L^2(\Bbb S^2)$, we say that $\operatorname{degree}(f)>m$ if for any $n>m$ there holds $\pi_n f=0$. We also denote ${\psi(x)=(x_1+ix_2)^n}$ by the highest weight spherical harmonic functions.
	
For any $ \theta \in \R $, we define the rotation matrix $ R_{\theta} $ on $ \R^3 $ and the operator $ R_{\theta}^{*} $ on $ L^2(\Bbb S^2) $ by
\begin{equation*}
R_{\theta} := \left(
\begin{aligned}
& \cos \theta & ~~ -\sin \theta & \quad 0 ~~ \\
& \sin \theta & ~~ \cos \theta ~ & \quad 0 ~~ \\
& \quad 0 & 0 \quad & \quad 1 ~~
\end{aligned}\right)
\qquad \text{and} \qquad
\begin{aligned}
R_{\theta}^{*} : L^2(\Bbb S^2) & ~ \to L^2(\Bbb S^2) \\
f & ~ \mapsto R_{\theta}^{*}(f)(x) := f(R_{\theta}(x)).
\end{aligned}
\end{equation*}
Then, $ R_{\theta}^{*} $ is an unitary operator on $ L^2(\Bbb S^2) $ from the two facts below that $ \det(R_{\theta}) \equiv 1 $ and
\begin{equation*}
\begin{aligned}
\langle R_{\theta}^{*}(f), R_{\theta}^{*}(g) \rangle
= & ~ \int_{\Bbb S^2} R_{\theta}^{*}(f)(x) \overline{R_{\theta}^{*}(g)(x)} ~ dx
= \int_{\Bbb S^2} f(R_{\theta}(x)) \overline{g(R_{\theta}(x))} ~ dx \\
= & ~ \int_{\Bbb S^2} f(y) \overline{g(y)} ~ \vert \det(R_{\theta}) \vert^{-1} dy
= \langle f, g \rangle.
\end{aligned}
\end{equation*}
	
We recall a decomposition lemma of Burq-G\'erard-Tzvetkov \cite{Burq2001}, in which the operator $R^*_\theta$ is used to exclude low frequency for a given function $f\in L^2(\Bbb S^2)$.

\begin{lemma}[Decomposition of function, \cite{Burq2001}]\label{lem-rotation}
Let $f\in L^2(\Bbb S^2)$ such that for all $\theta\in\R$ there holds
\begin{equation*}
R^*_\theta f = e^{in\theta} f.
\end{equation*}
Then, there exists $\alpha_1\in\Bbb C$ and $g$ with $\pi_m g=0$ for all $m>n$ such that
\begin{equation*}
f = \alpha_1 \psi + g.
\end{equation*}
\end{lemma}
	
For $ \psi(x)=(x_1+ix_2)^n $, one can compute that
\begin{equation*}
\begin{aligned}
R_{\theta}^{*}(\psi)(x)= & ~ \psi(R_{\theta}(x))= \psi(x_1 \cos \theta - x_2 \sin \theta, x_1 \sin \theta + x_2 \cos \theta, x_3) \\
= & ~ [(x_1 \cos \theta - x_2 \sin \theta) + i (x_1 \sin \theta + x_2 \cos \theta)]^n \\
= & ~ [x_1 (\cos \theta + i \sin \theta) + i x_2 (\cos \theta + i \sin \theta)]^n= e^{in\theta} \psi(x).
\end{aligned}
\end{equation*}
Furthermore, let $f(u):= |u|^{\alpha-1} u$, one can verify that 
\begin{equation*}
R_{\theta}^{*}(f(\psi))(x) = |\psi(R_{\theta}(x))|^{\alpha-1} \psi(R_{\theta}(x)) = |e^{in\theta} \psi(x)|^{\alpha-1} (e^{in \theta} \psi(x)) = e^{in\theta}(f(\psi))(x).
\end{equation*}
By Lemma \ref{lem-rotation}, there exists $ \alpha_1 \in \Bbb C $ and $ g $ with $\operatorname{degree}(g)>n$ such that
\begin{equation*}
f(\psi) = \alpha_1 \psi + g,
\end{equation*}
with
\begin{equation*}
\alpha_1 = \frac{\|\psi\|^{\alpha+1}_{L^{\alpha+1}}}{\|\psi\|^2_{L^{2}}}.
\end{equation*}
	
As a further application of Lemma \ref{lem-rotation}, we can decompose the solution to \eqref{fml-NLS} with initial data which belongs to highest weight spherical harmonic functions.
\begin{corollary}\label{cor-decomp}
Let $\alpha_1\in\Bbb C$ and $\psi(x)=(x_1+ix_2)^n$ and assume $u\in C([-T,T];H^s(\Bbb S^2))$ satisfies \eqref{fml-NLS} with $u_0 = \alpha_1\psi$ for some $\alpha_1>0$. Then, for every $t\in [-T,T]$
\begin{equation*}
u(t,x) = \alpha_1(t)\psi + r(t),
\end{equation*}
with $\operatorname{deg}(r)>n$.
\end{corollary}

Using the above decomposition, we construct the ansatz to \eqref{fml-NLS} as follows.
\begin{proposition}\label{prop-ansatz}
Let $T>0$, $\kappa\in(0,1)$, $s<\frac{1}{4}$. For $n\in\N$, denote that $\phi_n(x)=n^{\frac{1}{4}-s}\psi(x)$ and $\omega_n = \|\phi_n\|_{L^{\alpha+1}(\Bbb S^2)}/\|\phi_n\|_{L^2(\Bbb S^2)} \sim n^{(\alpha-1)(\frac{1}{4}-s)}$. Then, suppose that 
\begin{equation*}
\lim\limits_{n\to\infty} \omega_n T_n = \infty,
\end{equation*}
the solution $u_n(t)$ to \eqref{fml-NLS} with initial data $\kappa\phi_n(x)$ can be decomposed into
\begin{equation}\label{fml-ansa}
u(t,x) = \kappa \exp\big( -it(n(n+1)+\kappa^{\alpha-1}\omega_n) \big)(\phi_n+r_n(t)),\quad \forall t\in[0,T_n].
\end{equation}
Here, the remainder $r_n(t)$ satisfies
\begin{equation}\label{fml-ansa-rem}
\sup_{0\le t\le T_n}\|r_n(t)\|_{H^s(\Bbb S^2)}\to0,\,\,\text{as}\,\, n\to \infty.
\end{equation}
Moreover, there exists uniform constant $C>0$ such that
\begin{equation}\label{fml-Hs-bound}
\|u_n\|_{L^\infty(\R;H^s(\Bbb S^2))} \le C\kappa.
\end{equation}
\end{proposition}

\begin{proof}
For convenience, we abbreviate $u_n$, $\phi_n$ and $\omega_n$ to $u$, $\phi$ and $\omega$ respectively. 
		
We prove \eqref{fml-Hs-bound} first. Thanks to the mass and energy conservation laws \eqref{mass} and \eqref{energy}, we have
\begin{equation*}
\|u\|_{L^\infty(\R;L^2(\Bbb S^2))} = \kappa\|\phi\|_{L^2(\Bbb S^2)}\lesssim \kappa n^{-s}, 
\end{equation*}
and
\begin{equation*}
\|u\|_{L^\infty(\R;\dot{H}^1(\Bbb S^2))} \lesssim \kappa\big(\|\phi\|_{\dot{H}^1(\Bbb S^2)}+\|\phi\|_{L^{\alpha+1}(\Bbb S^2)}\big)\lesssim \kappa n^{1-s}.
\end{equation*}
By interpolation, we can find
\begin{equation*}
\|u\|_{L^\infty(\R;\dot{H}^s(\Bbb S^2))} \lesssim \|u\|^{1-s}_{L^\infty(\R;L^2(\Bbb S^2))} \|u\|^{s}_{L^\infty(\R;\dot{H}^1(\Bbb S^2))} \lesssim \kappa,
\end{equation*}
which implies  \eqref{fml-Hs-bound}.
		
It remains to prove \eqref{fml-ansa-rem}.  For convenience, we denote by
\begin{equation*}
h(t) = \exp(it\big( n(n+1)+\kappa^{\alpha-1}\omega \big)),
\end{equation*}
with
\begin{equation*}
\omega = \frac{\|\phi\|^{\alpha+1}_{L^{\alpha+1}(\Bbb S^2)}}{\|\phi\|^2_{L^2(\Bbb S^2)}} \sim n^{(\alpha-1)(\frac{1}{4}-s)}.
\end{equation*}
Thus, $h(t)$ satisfies
\begin{equation*}
\begin{cases}
ih_t + n(n+1) h = \kappa^{\alpha-1} \omega |h|^{\alpha-1}h,\\
h(0)=1.
\end{cases}
\end{equation*}
We write the solution to \eqref{fml-NLS} of the form 
\begin{align}\label{equivalent-1}
u(t) = \kappa h(t)(\phi + w(t)).
\end{align} From Corollary \ref{cor-decomp}, we find that $w(t)$ satisfies $w(0)=0$ and
\begin{equation}\label{fml-w-eq}
iw_t + (\Delta_g+n(n+1)+\kappa^{\alpha-1}\omega) w = \kappa^{\alpha-1}\big( |\phi+w|^{\alpha-1}(\phi+w)-|\phi|^{\alpha-1}\phi+r \big),
\end{equation}
where $\operatorname{deg}r > n$. We can verify that for every $\theta\in\R$
\begin{equation*}
R^*_\theta( |\phi+w|^{\alpha-1}(\phi+w)-|w|^{\alpha-1}w+r) =  e^{in\theta}(|\phi+w|^{\alpha-1}(\phi+w)-|w|^{\alpha-1}w+r).
\end{equation*}
By Corollary \ref{cor-decomp}, we conclude that
\begin{equation*}
w(t) = z(t)\phi+q(t),
\end{equation*}
with $z(t)\in\Bbb C$ and $\operatorname{deg}q(t) > n$. 
		
Next, we turn to control $\sup_{t\in[0,T]}\|q(t)\|_{H^s(\Bbb S^2)}$ and $\sup_{t\in[0,T]}|z(t)|$ with fixed $T>0$.
		
\textbf{Bound for $\sup_{t\in[0,T]}\|q(t)\|_{H^s(\Bbb S^2)}$.} Since $q(t)$ is orthogonal to $\phi$ in $\dot{H}^k(\Bbb S^2)$ for all $k\in\N$, inserting \eqref{equivalent-1} to  \eqref{mass} and \eqref{energy},  we deduce
\begin{equation}\label{fml-mass-zq}
|1+z(t)|^2\|\phi\|^2_{L^2(\Bbb S^2)}+\|q(t)\|^2_{L^2(\Bbb S^2)} = \|\phi\|^2_{L^2(\Bbb S^2)},  
\end{equation} 
and
\begin{equation}\label{fml-energy-zq}
\begin{aligned}
&|1+z(t)|^2\|\phi\|^2_{\dot{H}^1(\Bbb S^2)} + \|q(t)\|^2_{\dot{H}^1(\Bbb S^2)}+ \frac{2}{(\alpha+1)\kappa^2}\|u(t)\|^{\alpha+1}_{L^{\alpha+1}(\Bbb S^2)}= \|\phi\|^2_{\dot{H}^1(\Bbb S^2)}+\frac{2\kappa^{\alpha-1}}{\alpha+1}\|\phi\|^{\alpha+1}_{L^{\alpha+1}(\Bbb S^2)},
\end{aligned}
\end{equation}
respectively. Notice that $\phi\in E_n$, thus for all $k\in\N$
\begin{equation*}
\|\phi\|^2_{\dot{H}^k(\Bbb S^2)} = n^k(n+1)^k\|\phi\|^2_{L^2(\Bbb S^2)}.
\end{equation*}
Using this fact to \eqref{fml-mass-zq} and \eqref{fml-energy-zq}, we can eliminate the terms contain $\|\phi\|^2_{L^2(\Bbb S^2)}$, thus we obtain
\begin{equation*}
\begin{aligned}
\|q(t)\|^2_{\dot{H}^1(\Bbb S^2)}-n^2(n+1)^2\|q(t)\|^2_{L^2(\Bbb S^2)} &\le \frac{2\kappa^{\alpha-1}}{\alpha+1}\|\phi\|^{\alpha+1}_{L^{\alpha+1}(\Bbb S^2)}\lesssim n^{(\alpha-1)(\frac{1}{4}-s)-2s}.
\end{aligned}
\end{equation*}
From spectral resolution, we expand $q(t)$ as
\begin{equation*}
q(t) = \sum_{m>n} q_m(t),
\end{equation*}
where $\operatorname{degree}q_m(t) = m$. Therefore, it gives that
\begin{align*}
&\|q(t)\|^2_{\dot{H}^1(\Bbb S^2)}-n^2(n+1)^2\|q(t)\|^2_{L^2(\Bbb S^2)}=\sum_{m\ge n+1}\big[ m(m+1)-n(n+1) \big]\|q_m(t)\|^2_{L^2(\Bbb S^2)}
\end{align*}
		
A direct calculation implies that there exists a constant  $C$ independent of $n$ such that
\begin{equation*}
\begin{aligned}
m(m+1)-n(n+1) &\ge C n,\,\,\frac{m(m+1)-n(n+1)}{m(m+1)} \ge C n^{-1}.
\end{aligned}
\end{equation*}
Thus, we get
\begin{equation}\label{fml-qt-bound}
\|q(t)\|^2_{L^2(\Bbb S^2)} \lesssim n^{\frac{\alpha-5}{4}-(\alpha+1)s},\,\, \|q(t)\|^2_{\dot{H}^1(\Bbb S^2)} \lesssim n^{\frac{\alpha+3}{4}-(\alpha+1)s}.
\end{equation}
By interpolation, we conclude
\begin{equation*}
\|q(t)\|_{\dot{H}^s(\Bbb S^2)} \lesssim \|q(t)\|^{1-s}_{L^2(\Bbb S^2)} \|q(t)\|^{s}_{\dot{H}^1(\Bbb S^2)} \lesssim n^{\frac{\alpha-5}{8}-(\alpha-1)\frac{s}{2}},
\end{equation*}
which is acceptable.
		
\textbf{Bound for $\sup_{t\in[0,T]}|z(t)|$}. Since $r(t)$ and $q(t)$ are orthogonal to $\phi$,  multiplying $\phi$ and taking the inner product  on both sides of \eqref{fml-w-eq}, we obtain
\begin{equation}\label{fml-z-eq}
\begin{aligned}
iz_t+\omega\kappa^{\alpha-1}z &= \kappa^{\alpha-1}\|\phi\|^{-2}_{L^2(\Bbb S^2)} \big( \langle F(\phi+w)-F(\phi),\phi\rangle\big).
\end{aligned}
\end{equation}
Note that $\alpha\geq3$ and $|\phi+w|^{\alpha-1}$ has second-order smoothness, then we expand it as follows:
\begin{equation*}
\begin{aligned}
|\phi+w|^{\alpha-1} &= |\phi|^{\alpha-1} + \frac{(\alpha-1)}{2}|\phi|^{\alpha-3}(\phi+\overline{\phi})w + (\alpha-1)(\alpha-2)\int_0^1 (1-s)|\phi+sw|^{\alpha-3}w^2\, ds\\
&= |\phi|^{\alpha-1} + \frac{(\alpha-1)}{2}|\phi|^{\alpha-3}(\phi+\overline{\phi})w + \mathcal{O}(|\phi|^{\alpha-3}|w|^2 + |w|^{\alpha-1}).
\end{aligned}
\end{equation*}
Recall that $w(t) = z(t)\phi+q(t)$ and $|\phi|^{\alpha-1}\phi = \omega\phi+r$, we rewrite the nonlinearity satisfies \eqref{fml-z-eq} by
\begin{equation*}
\begin{aligned}
&\Big\langle \big(|\phi+w|^{\alpha-1}(\phi+w) - |\phi|^{\alpha-1}\phi \big),\phi\Big\rangle \\
=&\|\phi\|^2_{L^2(\Bbb S^2)}\Big( \frac{\alpha+1}{2}\omega z+\frac{\alpha-1}{2}\omega \overline{z} \Big) \\
&+ O\Big((|z|^2+|z|^3+|z|^{\alpha-1})\|\phi\|^{\alpha+1}_{L^{\alpha+1}(\Bbb S^2)} + \int_{\Bbb S^2}(|q|^3|\phi|^{\alpha-2}+|q|^2|\phi|^{\alpha-1}+|q|^{\alpha-1}|\phi|^2+|q||r|)d\sigma(x)\Big).
\end{aligned}
\end{equation*}
Since $\alpha-1\in (2,3)$, $|z|^{p-1}$ can be absorbed into $|z|^2+|z|^3$. Moreover, we also get 
\begin{equation*}
\int_{\Bbb S^2}|q|^{\alpha-1}|\phi|^2\, dx \lesssim \int_{\Bbb S^2}|q|^3|\phi|^{\alpha-2} + |q|^2|\phi|^{\alpha-1}\, dx.
\end{equation*}
Therefore, \eqref{fml-z-eq} becomes
\begin{equation*}
\begin{aligned}
iz_t-\frac{\alpha-1}{2}\Big(\omega\kappa^{\alpha-1}z+\omega\kappa^{\alpha-1}\overline{z}\Big) =& \kappa^{\alpha-1}\|\phi\|^{-2}_{L^2(\Bbb S^2)} O\Big((|z|^2+|z|^3)\|\phi\|^{\alpha+1}_{L^{\alpha+1}(\Bbb S^2)} \\
&+ \int_{\Bbb S^2}(|q|^3|\phi|^{\alpha-2}+|q|^2|\phi|^{\alpha-1}+|q||r|)d\sigma(x)\Big).
\end{aligned}
\end{equation*}
Invoking the $L^2$ estimate of $q(t)$, we have
\begin{equation*}
\begin{aligned}
\kappa^{\alpha-1}\|\phi\|^{-2}_{L^2(\Bbb S^2)}\int_{\Bbb S^2}|q|^3|\phi|^{\alpha-2}d\sigma(x) &\lesssim n^{2s}\|q\|^{2}_{L^2(\Bbb S^2)}\|q\|_{\dot{H}^1(\Bbb S^2)}\|\phi\|^{\alpha-2}_{L^\infty(\Bbb S^2)}\\
&\lesssim n^{2s} n^{\frac{\alpha-5}{4}-(\alpha+1)s}n^{\frac{\alpha+3}{8}-\frac{(\alpha+1)s}{2}}n^{(\alpha-2)(\frac{1}{4}-s)}\\
&\lesssim n^{\frac{5\alpha-11}{8}-\frac{5}{2}(\alpha-1)s},
\end{aligned}
\end{equation*}
and
\begin{equation*}
\begin{aligned}
\kappa^2\|\phi\|^{-2}_{L^2(\Bbb S^2)}\int_{\Bbb S^2}|q|^2|\phi|^{\alpha-1}d\sigma(x) &\lesssim n^{2s}\|q\|^2_{L^2(\Bbb S^2)}\|\phi\|^{\alpha-1}_{L^\infty(\Bbb S^2)}
\lesssim n^{\frac{\alpha-3}{2}-2(\alpha-1)s}. 
\end{aligned}
\end{equation*} 
Note that $|\phi|^{\alpha-1}\phi = \omega\phi+r$, multiplying $r$ and taking the  inner product implies that
\begin{equation*}
\|r\|^2_{L^2(\Bbb S^2)} = \int_{\Bbb S^2} |\phi|^{\alpha-1}\phi \overline{r} d\sigma(x) \lesssim \|\phi\|^\alpha_{L^{2\alpha}(\Bbb S^2)} \|r\|_{L^2(\Bbb S^2)}.
\end{equation*}
By using the eigenfunction estimate, we have 
\begin{equation*}
\|r\|_{L^2(\Bbb S^2)} \lesssim \|\phi\|^\alpha_{L^{2\alpha}(\Bbb S^2)} \lesssim n^{\frac{\alpha-1}{4}-s\alpha}.
\end{equation*}
Consequently, the last source term enjoys the bound
\begin{equation*}
\begin{aligned}
\kappa^2\|\phi\|^{-2}_{L^2(\Bbb S^2)}\int_{\Bbb S^2}|q||r|d\sigma(x) &\lesssim n^{2s}\|q\|_{L^2(\Bbb S^2)}\|r\|_{L^2(\Bbb S^2)}\\
&\lesssim n^{2s}n^{\frac{\alpha-5}{8}-\frac{(\alpha+1)s}{2}}n^{\frac{\alpha-1}{4}-s\alpha}\\
&\lesssim n^{\frac{3\alpha-7}{8}-\frac{3(\alpha-1)}{2}s}. 
\end{aligned}
\end{equation*}
Therefore ,we rewrite the equation for $z(t)$ with initial data $z(0)=0$ as follows
\begin{align*}
iz_t=&(\alpha-1)\omega \kappa^{\alpha-1} \Re(z) +\mathcal{O}(\omega|z|^2+\omega|z|^3) \\
&\hspace{20ex}+ 
\begin{cases}
\mathcal{O}\big(n^{\frac{3\alpha-7}{8}-\frac{3\alpha-1)}{2}s}\big),\quad s\in\big(\frac{1}{4}-\frac{1}{4(\alpha-1)},\frac{1}{4}\big)\\
\mathcal{O}\big( n^{\frac{5\alpha-11}{8}-\frac{5}{2}(\alpha-1)s} \big),\quad s\in\big(0,\frac{1}{4}-\frac{1}{4(\alpha-1)}\big] .
\end{cases}
\end{align*} 
First, we  consider the case  $s>\frac{1}{4}-\frac{1}{4(\alpha-1)}$. From mass conservation \eqref{fml-mass-zq}, we estimate
\begin{equation*}
\Big| \frac{\alpha-1}{2}|z|^2+(\alpha-1)\Re(z)\Big|=\frac{\alpha-1}{2}\Big| 1-|1+z|^2 \Big|=\frac{\alpha-1}{2}\|q\|^2_{L^2(\Bbb S^2)}\|\phi\|^{-2}_{L^2(\Bbb S^2)} = \mathcal{O}(n^{\frac{\alpha-5}{4}-(\alpha-1)s}).
\end{equation*}
It follows that $(\alpha-1)\Re(z)\omega \lesssim n^{\frac{\alpha-3}{2}-2(\alpha-1)s}$, which can be absorbed into $n^{\frac{3\alpha-7}{8}-\frac{3(\alpha-1)}{2}s}$. Hence,  $z(t)$ satisfies 
\begin{equation}\label{fml-z-eq-refined}
iz_t=\mathcal{O}(\omega|z|^2+\omega|z|^3+n^{\frac{3\alpha-7}{8}-\frac{3(\alpha-1)}{2}s}),
\end{equation}
with $z(0)=0$. Let $0<\varepsilon\ll1$, we take
\begin{equation*}
T_n = n^{\frac{5(\alpha-1)}{4}s-\frac{5\alpha-9}{16}-\varepsilon},
\end{equation*}
and denote that
\begin{equation*}
M(T) := \sup_{t\in[0,T]}|z(t)|.
\end{equation*}
It is clear that
\begin{equation*}
M(T_n) \lesssim T_n\big(\omega M^2(T_n)+\omega M^3(T_n) + n^{\frac{3\alpha-7}{8}-\frac{3(\alpha-1)}{2}s}\big).
\end{equation*}
Next, we define 
\begin{equation*}
\widetilde{M}(T_n) := T_n^{-1}n^{\frac{3(\alpha-1)}{2}s-\frac{3\alpha-7}{8}} M(T_n).
\end{equation*}
Recall that $\omega=\mathcal{O}(n^{(\alpha-1)(\frac{1}{4}-s)})$ and we arrive at
\begin{equation*}
\widetilde{M}(T_n) \lesssim \big(T_n^2n^{\frac{5\alpha-9}{8}-\frac{5(\alpha-1)}{2}s}\widetilde{M}^2(T_n)+T_n^3n^{(\alpha-2)-4(\alpha-1)s}M^3(T_n) + 1\big), 
\end{equation*}
with $M(0)=0$. From the choice of $T_n$, we have $\widetilde{M}(T_n)$ is uniformly bounded with respect to $n$. Thus, $M(T_n)$ enjoys the bound
\begin{equation*}
\sup_{0\le T\le T_n}M(T) \lesssim n^{-\frac{\alpha-1}{4}s-\frac{\alpha-5}{16}-\varepsilon}.
\end{equation*}
		
Using the bound for $q(t)$ and $z(t)$, we conclude that for some $\delta>0$
\begin{equation*}
\|r_n(t)\|_{\dot{H}^s(\Bbb S^2)} \lesssim \sup_{t\in[0,T]}|z(t)|\|\phi\|_{\dot{H}^s(\Bbb S^2)} + \|q(t)\|_{\dot{H}^s(\Bbb S^2)} \lesssim n^{-\delta}.
\end{equation*}
Since 
\begin{equation*}
\lim\limits_{n\to\infty}\omega_n T_n = \lim\limits_{n\to\infty}n^{\frac{5-\alpha}{16}+\frac{\alpha-1}{4}s-\varepsilon} =\infty,
\end{equation*}
the choice for $T_n$ is acceptable.
		
Now, we turn to $0<s\le \frac{1}{4}-\frac{1}{4(\alpha-1)}$. We may take
\begin{equation*}
T_n = n^{\frac{7}{4}(\alpha-1)s-\frac{7\alpha-13}{16}-\varepsilon},
\end{equation*}
and 
\begin{equation*}
\widetilde{M}(T_n) := T_n^{-1}n^{\frac{5}{2}(\alpha-1)s-\frac{5\alpha-11}{8}} M(T_n).
\end{equation*}
Similarly for $s>\frac{1}{4}-\frac{1}{2(\alpha-1)}$, we obtain
\begin{equation*}
M(T_n)\lesssim n^{-\frac{3}{4}(\alpha-1)s+\frac{3(\alpha-3)}{16}-\varepsilon} \lesssim n^{-\delta},
\end{equation*}
and
\begin{equation*}
\lim\limits_{n\to\infty}\omega_n T_n = \lim\limits_{n\to\infty}n^{\frac{3}{4}(\alpha-1)s-\frac{3(\alpha-3)}{16}+\varepsilon} =\infty.
\end{equation*}
\end{proof}
	
\begin{proof}[Proof of Theorem \ref{thm-illpose2}]
Let $\kappa_n,\kappa_n^\prime\in (0,2)$ to be chosen later and $u_{0,n} = \kappa_n\phi_n(x),u^\prime_{0,n} = \kappa_n^\prime\phi_n(x)$. From Proposition \ref{prop-ansatz}, the solution $u_n$ and $u_n^\prime$ to \eqref{fml-NLS} with initial data $u_{0,n}$ and $u_{0,n}^\prime$ has the form
\begin{equation*}
\begin{aligned}
u_n = \kappa_n\exp\big(-it(n(n+1)+\kappa_n^{\alpha-1}\omega_n)\big) (\phi_n+r_n),\\
u^\prime_n = \kappa^\prime_n\exp\big(-it(n(n+1)+(\kappa^\prime_n)^{\alpha-1}\omega_n)\big) (\phi_n+r^\prime_n).
\end{aligned}
\end{equation*}
		
We take $\kappa_n\equiv1$, $\kappa_n^\prime = 1+\pi (\omega_n T_n)^{2-\alpha}$ where $T_n$ satisfies Proposition \ref{prop-ansatz}. Then, direct calculus yields
\begin{equation*}
\begin{aligned}
\|u^\prime-u\|_{L^\infty([0,\delta];\dot{H}^s(\Bbb S^2))} &\ge \Big| \kappa_n\exp\big(-it(n(n+1)+\kappa_n^2\omega_n)\big)-\kappa^\prime_n\exp\big(-it(n(n+1)+(\kappa^\prime_n)^2\omega_n)\big) \Big|\\
&\hspace{4ex}- \|r_n\|_{L^\infty([0,\delta];\dot{H}^s(\Bbb S^2))}-\|r^\prime_n\|_{L^\infty([0,\delta];\dot{H}^s(\Bbb S^2))}\\
&\ge \kappa_n\Big| \exp\big(-it(n(n+1)+\kappa_n^2\omega_n)\big)-\exp\big(-it(n(n+1)+(\kappa^\prime_n)^2\omega_n)\big) \Big|\\
&\hspace{4ex}-|\kappa_n-\kappa_n^\prime|-o(1)\\
&\ge 1-\pi (\omega_n T_n)^{2-\alpha}-o(1).
\end{aligned}
\end{equation*}
Since $s\in (0,\frac{1}{4})$, we can take sufficiently large $n$ such that $1-\pi\delta^{-1}n^{2s-\frac{1}{2}}-o(1)\sim\frac{k}{2}$ for some large $k\in 2\mathbb{N}+1$. Consequently, we verify $(3)$ in Theorem \ref{thm-illpose2}, which completes the proof of Theorem \ref{thm-illpose2}.
\end{proof}


\section*{Acknowledgments}
 J. Zheng was supported by National key R\&D program of China: 2021YFA1002500 and  NSF grant of China (No. 12271051).

\end{document}